\documentclass{amsart}
\usepackage{amssymb,cite}

\advance\textheight by \topskip
\newtheorem{thm}{Theorem}[section]
\newtheorem{prop}[thm]{Proposition}
\newtheorem{lem}[thm]{Lemma}

\newcommand{\rmand}{\quad\hbox{ and }\quad}

\makeatletter \@addtoreset{equation}{section} \makeatother

\usepackage{xcolor}

\title[On colored set partitions of type $B_n$]{On colored set partitions of type $B_n$}

\author{David G.L. Wang}
\address{Department of Mathematics, University of Haifa,
199 Aba Khoushy Ave. Mount Carmel, Haifa, 3498838, Israel}
\email{david.combin@gmail.com, wgl@math.haifa.ac.il}



\begin{document}
\maketitle

\begin{abstract}
Generalizing Reiner's notion of set partitions of type $B_n$, we define colored $B_n$-partitions by coloring the elements in and not in the zero-block respectively. Considering the generating function of colored $B_n$-partitions, we get the exact formulas for the expectation and variance of the number of non-zero-blocks in a random colored $B_n$-partition. We find an asymptotic expression of the total number of colored $B_n$-partitions up to an error of $O(n^{-1/2}\log^{7/2}{n})$, and prove that the centralized and normalized number of non-zero-blocks is asymptotic normal over colored $B_n$-partitions.
\end{abstract}

\section{Introduction}\label{sec:intro}

It is well known that the centralized and normalized 
number of blocks over ordinary set partitions
is asymptotic normal, which is due to
Harper~\cite{Har67}.
There are various studies on the asymptotic behavior of 
set partitions. Basic notions and results
of limiting distributions can be found in~\cite{Sac97B,FS09B,Man13B}.
In this paper,
we focus on the limiting distribution of another set partition structure.

A partition of the set~$[n]=\{1,2,\ldots,n\}$ is 
a collection of pairwise disjoint sets $b_1$, $b_2$, $\ldots$, $b_m$ 
such that $\cup_{i=1}^mb_i=[n]$. 
The sets $b_i$
are called {\em blocks}.
In 1997, Reiner~\cite{Rei97} introduced the notion of {\em set partition of type $B_n$} ({\em $B_n$-partition} for short),
which was defined to be a partition~$\pi$ of the set
\[
[\pm n]=\{\,-n,\,\ldots,\,-2,\,-1,\,1,\, 2,\, \ldots,\, n\,\}
\]
such that for any block~$b$ of~$\pi$, $-b$~is also a
block of~$\pi$, and that $\pi$ contains at most one block $b$ 
satisfying $b=-b$. 
The block $b=-b$, if it exists, is called the {\em zero-block}.
We call the pair $\pm b$ of blocks 
a {\em block pair} of~$\pi$ if~$b$ is not the zero-block.

Most attention that has been paid to $B_n$-partitions are devoted to its lattice structure.
One of the main reasons seems to be the development in subspace arrangements.
A main question regarding subspace arrangements 
is to study the structure of the complement 
subspace~$\mathcal{C}_\mathcal{A}=V-\cup_{X\in\mathcal{A}}X$,
where~$V$ is a finite-dimensional vector space, and
$\mathcal{A}$ a subspace arrangement of~$V$.
The Goresky-MacPherson formula~\cite{GM88B} allows one 
to translate the problem of determining the homology groups
of~$\mathcal{C}_\mathcal{A}$ into the problem of studying certain groups 
related to the intersection lattice.
In fact,
the lattice of ordinary set partitions can be considered as the
intersection lattice for the hyperplane arrangement corresponding
to the root system of type~$A$, see~\cite{BB05,Hum90};
and the intersection lattice of the Coxeter arrangement of type~$B_n$ 
corresponds to $B_n$-partitions, see~\cite{BS96}.

Set partitions of type $B_n$ are also referred to as {\em signed partitions}.
In some papers, for example~\cite{Whi10D}, a $B_n$-partition is defined 
to be a partition of the set $[n]\cup\{0\}$ allowing barring
elements in~$[n]$ such that 
every element in the block containing the element $0$ is unbarred,
and that the minimum element in each block is unbarred.
The block containing $0$ is called the {\em zero-block}.
In fact, this definition for $B_n$-partitions is equivalent to the one given by Reiner.
This can be seen by identifying the block 
\[
\{0,\,a_1,\,a_2,\,\ldots,\,a_r\}
\]
containing $0$ with Reiner's zero-block 
\[
\{\pm a_1,\,\pm a_2,\,\ldots,\,\pm a_r\},
\]
and identifying the block 
\[
\{x_1,\,x_2,\,\ldots,\,x_s,\,\bar{y_1},\,\bar{y_2},\,\ldots,\,\bar{y_t}\}
\]
with Reiner's block pair 
\[
\pm\{x_1,\,x_2,\,\ldots,\,x_s,\,-y_1,\,-y_2,\,\ldots,\,-y_t\}.
\]
However, the wording ``the number of non-zero-blocks'' is therefore confusing
because one non-zero-block in the latter definition corresponds
to two non-zero-blocks in Reiner's definition.
To avoid this,
we will use the latter definition throughout.

Similar to unbarring the minimal elements, 
we generalize $B_n$-partitions by coloring its elements.
Let $n$, $m$ and $c$ be positive integers.
Let $\pi$ be a partition of the set~$[n]\cup\{0\}$.
Let $C_1$ be a list of~$c$ colors, and $C_2$ a list of~$m$ colors.
For any $x\in[n]\cup\{0\}$, let $b_x$ be the block of $\pi$ containing $x$.
Then~$b_0$ is the zero-block.
For any $x\in[n]$, we color $x$ in
\[
\begin{cases}
\text{the first color in $C_2$},& \text{if $x=\min b_x$};\\
\text{some color in $C_1$},& \text{if $x\ne\min b_x$ and $x\in b_0$};\\
\text{some color in $C_2$},& \text{if $x\ne\min b_x$ and $x\not\in b_0$}.
\end{cases}•
\]
In this way, every element $x\in[n]$ can be regarded as a pair $(x,cl)$,
where $cl$ denotes the color of~$x$.
We call the resulting partition a {\em $(c,m)$-colored $B_n$-partition},
and denote the set of $(c,m)$-colored $B_n$-partitions by $\Pi_{n,m,c}$.
For notation convenience, we can ignore 
the element~$0$ if $\pi$ contains the block~$\{0\}$.
For example, when $c=m=n=2$, 
if we denote $C_1=[c_1,c_2]$ and $C_2=[m_1,m_2]$,
then the set $\Pi_{2,2,2}$ consists of the following $11$ partitions:
\[
\begin{array}{llll}
(1,m_1)/(2,m_1) \quad
& (1,m_1),(2,m_1) \quad 
& (1,m_1),(2,m_2) \quad \\[5pt]
0,(1,c_1)/(2,m_1) \quad
& 0,(1,c_2)/(2,m_1) \quad
& 0,(2,c_1)/(1,m_1) \quad
& 0,(2,c_2)/(1,m_1) \quad\\[5pt]
0,(1,c_1),(2,c_1) \quad
& 0,(1,c_1),(2,c_2) \quad
& 0,(1,c_2),(2,c_1)\quad
& 0,(1,c_2),(2,c_2),
\end{array}
\]
where we use the slash symbol ``$/$'' to partition the blocks.

In particular,
the set of $(1,1)$-colored $B_n$-partitions has a transparent one to one 
correspondence with the set of ordinary partitions of $[n]\cup\{0\}$.
The $(1,2)$-colored $B_n$-partitions is another representation of Reiner's $B_n$-partitions.
More references can be found from Sloane's Online Encyclopedia of Integers Sequences (A039755).
In fact, as will be seen, 
the number of $(1,m)$-colored $B_n$-partitions with $k$~non-zero-blocks 
is the Whitney numbers $W_m(n,k)$ of the second kind, that is,
the number of elements of rank $k$ in the Dowling lattice $Q_n(G)$,
where $G$ is a finite group of order~$m$; see~\cite{Dow73,Ben96}.

In this paper, we will show that the centralized and normalized 
number of non-zero-blocks over colored $B_n$-partitions is asymptotic normal.
In particular, we get the asymptotic normality over set partitions of type~$B_n$.
A classical method to prove the normality of some limiting distribution is 
the following criterion; see~\cite[page 286]{Pit97} and~\cite{Ben73}.

Let $A=(a_0,a_1,\ldots,a_n)$ be a sequence of nonnegative numbers.
The sequence~$A$ is said to be a \emph{P\'olya frequency sequence} (PF sequence for short) of order~$r$,
if the infinite matrix $(a_{i-j})_{i,j\ge0}$ where $a_k=0$ for $k\notin\{0,1,\ldots,n\}$, is \emph{totally positive} of order $r$. 
For background on total positivity, see Karlin~\cite{Kar68B}.
It is well known that
$A$ is a PF sequence 
if the polynomial $A(z)=\sum_{k=0}^na_{k}z^k$ 
is either a positive constant or real-rooted; see~\cite{Sch55}.
Let $P$ denote the probability distribution on $\{0,1,\ldots,n\}$ defined by
normalization of~$A$. We call $P$ a PF distribution if~$A$ is a PF sequence.

\begin{thm}\label{thm:criterion}
Let~$(S_n)$ be a sequence of random variables such that $S_n$ has a PF distribution with mean $\mu_n$ and variance $\sigma_n^2$.
The asymptotic distribution of $(S_n-\mu_n)/\sigma_n$ is standard normal 
if and only if $\sigma_n\to\infty$ as $n\to\infty$.
\end{thm}

There are some standard methods to find the exact formulas 
for the expectation and variance in terms of~$A(1)$,
as well as many methods to prove the real-rootedness of certain polynomials; 
see~\cite{CW10,LW07} for example.
For the problem we are going to handle, we have $A(1)=|\Pi_{n,m,c}|$.
We will use the notation $T_{n}=|\Pi_{n,m,c}|$,
and determine an asymptotic formula for it
up to an error of $O(n^{-1/2}\log^{7/2}{n})$.
This enables us to prove that the variance tends to infinity.
More analytic tools for asymptotics
can be found in~\cite{deBru81B,FS09B,Odl95}.
We should mention that
to include full asymptotic expansions,
Harris and Schoenfeld~\cite{HS68} extended 
the theory of Hayman-admissibility~\cite{Hay56},
and this extended theory is applicable to functions of at least double ``exponential growth'',
which is the case that we treat in this paper.
See also~\cite{DGK05,GM06,SS99}.

In \S\ref{sec:RR:EV}, we give a recurrence relation 
of the cardinality $T_{n}$ 
and thus obtain the exact formulas of the expectation and variance of the
number of non-zero-blocks of a random partition in $\Pi_{n,m,c}$.
In \S\ref{sec:asymptotics}, 
we derive an asymptotic expression of $T_{n}$,
and prove that the centralized and normalized number of zero-blocks 
in a random $(c,m)$-colored $B_n$-partition is asymptotic normal by Theorem~\ref{thm:criterion}.

\section{The recurrence relation}\label{sec:RR:EV}

Let $T_{n,k}$ be the number of $(c,m)$-colored $B_n$-partitions with $k$ non-zero-blocks.
It is clear that $T_{n,0}=c^n$ and $T_{n,n}=1$.
For convenience, let $T_{n,k}=0$ if $k<0$ or $k>n$.
We claim that 
\begin{equation}\label{rec:T}
T_{n,k}=T_{n-1,k-1}+(mk+c)T_{n-1,k}.
\end{equation}
In fact, 
all $(c,m)$-colored $B_{n}$-partitions can be generated by inserting the element $n$ in some color
into a $(c,m)$-colored $B_{n-1}$-partition. Let $\pi$ be a $(c,m)$-colored $B_{n-1}$-partition. 
One may insert the element $n$ into $\pi$ as a singleton non-zero-block.
In this case, $n$ has to be colored in the given first color in $C_2$. 
This way generates all $(c,m)$-colored $B_{n}$-partitions such that~$n$ is a singleton non-zero-block. 
So it contributes the summand~$T_{n-1,k-1}$ to the recurrence~\eqref{rec:T}.
One may also insert $n$ into the zero-block of $\pi$, 
in any color in $C_1$. 
This way generates all $(c,m)$-colored $B_{n}$-partitions such that~$n$ is in the zero-block.
Therefore, it contributes the summand~$cT_{n-1,k}$ to~\eqref{rec:T}.
The last possibility is that we insert $n$ into a non-zero-block of $\pi$.
In this case, we have $k$ non-zero-blocks and $m$ colors in $C_2$ to choose.
So this way generates all $(c,m)$-colored $B_{n}$-partitions such that~$n$ is in a non-zero-block,
and contributes the summand~$mkT_{n-1,k}$ to~\eqref{rec:T}.
This proves the claim.

We mention that $T_{n,k}$ reduces to the Whitney number $W_m(n,k)$ 
of the second kind when $c=1$. This can be seen straightforwardly from the recurrence
\[
W_m(n,k) = W_m(n-1,k-1)+(mk+1)W_m(n -1,k),
\]
with $W_m(n,0)=W_m(n,1)=1$; 
see Benoumhani~\cite[Corollary 3]{Ben96}.

Consider the polynomial
$T_{n}(x)=\sum_{k=0}^nT_{n,k}\,x^k$.
For example,
\begin{align*}
T_{1}(x)&=x+c,\\
T_{2}(x)&=x^2+(m+2c)x+c^2,\\
T_{3}(x)&=x^3+(3m+3c)x^2+(m^2+3cm+3c^2)x+c^3.
\end{align*}
Multiplying~\eqref{rec:T} by $x^k$ and summing it over $k=1,2,\ldots,n-1$, we get
\[
T_{n}(x)=(x+c)T_{n-1}(x)+mxT_{n-1}'(x),\quad n\ge1.
\]

Then we have the following real-rootedness result.

\begin{thm}\label{thm:RR:T} 
For any $n\geq 1$,
the polynomial $T_{n}(x)$ has $n$ distinct negative roots.
\end{thm}

Theorem~\ref{thm:RR:T} can be shown by 
using a criterion for generalized Sturm sequences given 
by Liu and Wang~\cite[Corollary 2.4]{LW07}. In fact, they proved a more general result 
that any polynomial $P_n(x)$ of degree $n$ with nonnegative coefficients 
satisfying the recurrence relation
\[
P_n(x) = a_n(x)P_{n-1}(x) + b_n(x)P_{n-1}'(x),
\]
where $a_n(x)$ and $b_n(x)$ are real polynomials 
such that $b_n(x)\le0$ whenever $x\le0$, 
forms a generalized Sturm sequence,
and is therefore real-rooted.

Let $\xi_n$ be the random variable
denoting the number of non-zero-blocks of a partition in $\Pi_{n,m,c}$.
It follows that
\[
P(\xi_n=k)={T_{n,k}\over T_n}.
\]
Therefore, $\xi_n$ is the normalization of $(T_{n,0},\,T_{n,1},\,\ldots,\,T_{n,n})$ 
and thus has a PF distribution by Theorem~\ref{thm:RR:T}.
Denote the expectation of~$\xi_n$ by~$E_n$,
and the variance by~$V_n$.
It is a standard technique~\cite{Sac97B} to find the expectation and variance
by using the formulas
\[
E_n={T_{n}'(1)\over T_{n}}
\rmand
V_n=E_n-E_n^2+\frac{T_{n}''(1)}{T_{n}}.
\]
Note that both
\begin{align*}
T_{n}'(1)=\sum_{k=0}^n kT_{n,k}\quad\text{ and }\quad
T_{n}''(1)=\sum_{k=0}^n k(k-1)T_{n,k}
\end{align*}
can be expressed in terms of $T_{n}$ with the aid of the recurrence~\eqref{rec:T}.
Routine computations give the exact formulas for the expectation and variance of $\xi_n$.

\begin{thm}\label{thm_EV_M}
We have
\[
E_n={T_{n+1}\over mT_{n}}-{1+c\over m}
\rmand
V_n={T_{n+2}\over m^2T_{n}}-{T_{n+1}^2\over m^2T_{n}^2}-{1\over m}.
\]
\end{thm}

In view of Theorem~\ref{thm:criterion} and the above expression of $V_n$,
we are led to find an asymptotic expression for $T_{n}$, 
which will be dealt with in the next section.

\section{An Asymptotic expression of the total number $T_{n}$}\label{sec:asymptotics}

In this section,
we will find an asymptotic expression for the total cardinality $T_{n}$
by using the classical saddle point method,
which is due to Schr\"odinger~\cite{Sch44}.

First,
we seek the generating function 
\[
F(z)=\sum_{n\ge0}T_{n}\,{z^n\over n!}.
\]
In fact, for any $0\le k\le n$, one may choose~$k$ elements from~$[n]$
to form the non-zero-blocks of a colored $B_n$-partition, 
and let the other $n-k$ elements constitute the zero-block. 
Then, for any non-zero-block of size~$s$,
there are $m^{s-1}$ distinct colorings for that block.
And for the zero-block of size $n-k$,
there are $c^{n-k}$ ways of coloring its elements.
Therefore,
\[
T_{n}=\sum_{k=0}^n{n\choose k}c^{n-k}\sum_{j=0}^km^{k-j}S(k,j),
\]
where $S(k,j)$ is the Stirling numbers of the second kind.
By using the generating function
\[
\sum_{n}S(n,k){x^n\over n!}={1\over k!}\left(e^x-1\right)^k
\]
of Stirling numbers  (see~\cite[page 34]{Sta97}),
we deduce
\begin{equation}\label{egf:T}
F(z)=\exp\Bigl(\frac{e^{mz}-1}{m}+cz\Bigr).
\end{equation}

The {\em saddle point} of $F(z)$ is defined to be
the value $z$ that minimizes $z^{-n}F(z)$,
i.e., the unique positive solution $r$ of the equation
\[
r\bigl(e^{mr}+c\bigr)=n.
\]
Similar to the Lambert $W$ function $r=W(n)$ defined by $re^r=n$ (see~\cite{CGHJK96}),
it is easy to see that 
\begin{equation}\label{sim:r}
\begin{aligned}
r&={\log{n}\over m}\biggl(1+O\Bigl({\log\log{n}\over\log{n}}\Bigr)\biggr),\\
e^{mr}&={mn\over\log{n}}\biggl(1+O\Bigl({\log\log{n}\over\log{n}}\Bigr)\biggr).
\end{aligned}
\end{equation}
Note that
\[
\Bigl({n\over r}\Bigr)^n
=\exp\bigl(n\log(e^{mr}+c)\bigr)
=\exp\bigl(mrn+n\log(1+ce^{-mr})\bigr).
\]
Since $\log(1+x)=x+O(x^2)$ as $x\to 0$, we infer that
\begin{equation}\label{eq:n/r}
\Bigl({n\over r}\Bigr)^n
=e^{mrn+cr}\biggl(1+O\Bigl({\log^2n\over n}\Bigr)\biggr).
\end{equation}
Now we are in a position to give an estimate of $T_{n}$.

\begin{thm}\label{thm:sim:T}
We have
\[
T_{n}=\frac{1}{\sqrt{mr+1}}
\exp\biggl(mnr-n+\frac{n}{mr}+2cr-{1+c\over m}\biggr)
\cdot\biggl(1+O\Bigl(\frac{\log^{7/2}{n}}{\sqrt{n}}\Bigr)\biggr),
\]
where $r$ is the unique positive solution of the equation
$r\bigl(e^{mr}+c\bigr)=n$.
\end{thm}

The proof of Theorem~\ref{thm:sim:T} follows directly from the theory
of Hayman-admissibility, see~\cite{HS68}. 
In fact, it can be also directly proved by the saddle point method
with aid of~\eqref{egf:T} and~\eqref{eq:n/r}.

By Theorem~\ref{thm_EV_M} and Theorem~\ref{thm:sim:T}, 
it is easy to deduce that
\[
E_n\sim{n\over \log{n}}.
\]
For the sake of estimating $V_n$,
we need one more lemma.

\begin{lem}\label{lem1}
Let~$c$ be a nonnegative integer, and $f(x)=x\bigl(e^{mx}+c\bigr)$.
For any $i=0$, $1$, $2$,
let $t_i$ be the unique positive number such that $f(t_i)=n+i$.
Then we have
\begin{align}
&t_1-t_0={1\over mn}-{1\over m^2 n t_0}+O\Bigl({1\over n\log^2{n}}\Bigr),
\label{61}\\[5pt]
&t_2-t_1={1\over mn}-{1\over m^2 n t_0}+O\Bigl({1\over n\log^2{n}}\Bigr),
\label{62}\\[5pt]
&2t_1-t_0-t_2={1\over m n^2}+O\Bigl({1\over n^2\log{n}}\Bigr),\label{63}\\[5pt]
&{1\over t_0}+{1\over t_2}-{2\over t_1}=O\Bigl({1\over n^2\log^2{n}}\Bigr).\label{64}
\end{align}
\end{lem}

\begin{proof} First we prove~\eqref{61}.
By Cauchy's mean value theorem,
there exists $t\in(t_0,\,t_1)$ such that
\[
f(t_1)-f(t_0)=(t_1-t_0)f'(t).
\]
Since $f(t_1)-f(t_0)=1$ and $f'(t)=(mt+1)e^{mt}+c$, we have
\begin{equation}\label{56}
\frac{1}{(mt_1+1)e^{mt_1}+c}\le t_1-t_0\le\frac{1}{(mt_0+1)e^{mt_0}+c}.
\end{equation}
We can compute that
\[
\frac{1}{(mt_0+1)e^{mt_0}+c}-{1\over mn}+{1\over m^2nt_0}
={n+cm^2t_0^3-cmt_0^2\over m^2nt_0(mnt_0+n-cmt_0^2)}.
\]
Since $t_0=O(\log{n})$, we have
\begin{equation}\label{pf:lem1}
\frac{1}{(mt_0+1)e^{mt_0}+c}
={1\over mn}-{1\over m^2nt_0}+O\Bigl({1\over n\log^2{n}}\Bigr).
\end{equation}
On the other hand, we can compute
\[
\frac{1}{(mt_1+1)e^{mt_1}+c}-{1\over mn}+{1\over m^2nt_0}
={(n+1)(1+mt_1-mt_0)+m^2t_0t_1^2c-mt_1^2c-m^2t_0t_1
\over m^2nt_0(mnt_1+n+1+mt_1-cmt_1^2)}.
\]
By~\eqref{56} and~\eqref{pf:lem1}, 
we have $t_1-t_0=O(1/n)$.
Since $t_1=O(\log{n})$, we deduce that
\begin{equation}\label{pf:lem2}
\frac{1}{(mt_1+1)e^{mt_1}+c}
={1\over mn}-{1\over m^2nt_0}+O\Bigl({1\over n\log^2{n}}\Bigr).
\end{equation}
In view of~\eqref{56}, \eqref{pf:lem1} and~\eqref{pf:lem2},  
we infer that
\[
t_1-t_0={1\over mn}-{1\over m^2nt_0}+O\Bigl({1\over n\log^2{n}}\Bigr).
\]
Along the same line, we can prove that
\[
t_2-t_1={1\over mn}-{1\over m^2 n t_1}+O\Bigl({1\over n\log^2{n}}\Bigr).
\]
Since
\[
{1\over m^2 n t_1}-{1\over m^2 n t_0}
={t_0-t_1\over m^2nt_0t_1}
=O\Bigl({1\over n^2\log^2 n}\Bigr),
\]
we get the estimate~\eqref{62}.
We pause here for the following proposition,
which will imply the last two approximations.

\begin{prop}\label{prop:tech}
Let $h(x)$ be a continuous function defined on the closed
interval $[a,b]$. Suppose that $h''(x)$ exists
in the open interval $(a,b)$. Then for any $c\in(a,b)$,
there exists $s\in(a,b)$ such that
\begin{equation}\label{eq}
{h(a)\over(a-b)(a-c)}+{h(b)\over(b-a)(b-c)}
+{h(c)\over(c-a)(c-b)}={h''(s)\over 2}.
\end{equation}
\end{prop}

\begin{proof}
Let
\begin{align*}
f_1(x)&=(a-b)h(x)+(b-x)h(a)+(x-a)h(b),\\[5pt]
g_1(x)&=(a-b)(x-a)(x-b).
\end{align*}
Then the left hand side of~\eqref{eq}
becomes $f_1(c)/g_1(c)$. Note that $f_1(a)=g_1(a)=0$.
By Cauchy's mean value theorem, there exists $s_1\in(a,c)$ such that
\[
{f_1(c)\over g_1(c)}={f_1(c)-f_1(a)\over g_1(c)-g_1(a)}={f_1'(s_1)\over g_1'(s_1)}
={f_2(b)-f_2(a)\over g_2(b)-g_2(a)},
\]
where $f_2(x)=h(x)-h'(s_1)x$ and $g_2(x)=x^2-2s_1x$.
Again, by Cauchy's mean value theorem,
there exist $s_2\in(a,b)$ and $s\in(a,b)$ such that
\[
{f_2(a)-f_2(b)\over g_2(a)-g_2(b)}={f_2'(s_2)\over g_2'(s_2)}
={h'(s_2)-h'(s_1)\over 2s_2-2s_1}={h''(s)\over 2}.
\]
This completes the proof of Proposition~\ref{prop:tech}.
\end{proof}

Now we continue the proof of Lemma~\ref{lem1}. 
Set $h(x)=f(x)$, $a=t_0$, $b=t_2$, and $c=t_1$ in Proposition~\ref{prop:tech}. 
Multiplying the formula~\eqref{eq} by $(t_1-t_0)(t_2-t_1)(t_2-t_0)$ gives
\[
2t_1-t_0-t_2={1\over 2}f''(s)(t_1-t_0)(t_2-t_1)(t_2-t_0).
\]
In view of~\eqref{61} and~\eqref{62}, we deduce~\eqref{63} immediately from the above formula.
The estimate~\eqref{64} can be derived along the same line, 
by taking 
\[
h(x)={1\over x}\exp\biggl({m\over x}+c\biggr),\qquad
a={1\over t_2},\qquad
b={1\over t_0},\qquad
c={1\over t_1}
\] 
in Proposition~\ref{prop:tech}.
This completes the proof of Lemma~\ref{lem1}.
\end{proof}

Here is the main result of this paper.
\begin{thm}\label{thm:NLD:xi}
The random variable
$(\xi_n-E_n)/\sqrt{V_n}$
has an asymptotically standard normal distribution as~$n$~tends to
infinity. 
\end{thm}

\begin{proof}
By Theorem~\ref{thm:criterion} and Theorem~\ref{thm:RR:T}, 
it suffices to prove 
$\lim_{n\to\infty}V_n=\infty$, that is,
\[
\lim_{n\to\infty}{1\over m^2}\biggl({T_{n+2}\over T_{n}}-{T_{n+1}^2\over T_{n}^2}-m\biggr)=\infty.
\]
For $i=0,1,2$, suppose that
\[
t_i(e^{mt_i}+c)=n+i.
\]
By Theorem~\ref{thm:sim:T}, we have
\[
T_{n+i}=\frac{1}{\sqrt{mt_i+1}}
\exp\biggl(m(n+i)t_i-(n+i)+\frac{n+i}{mt_i}+2ct_i-{1+c\over m}\biggr)
\cdot\biggl(1+O\Bigl(\frac{\log^{7/2}{n}}{\sqrt{n}}\Bigr)\biggr).
\]
Then we can compute 
\[
{T_{n+2}\over T_{n}}
=\sqrt{mt_0+1\over mt_2+1}e^A\cdot\biggl(1+O\Bigl({\log^{7/2}{n}\over\sqrt{n}}\Bigr)\biggr),
\]
where 
\begin{equation}\label{pf:A}
A=2mt_2+{2\over m t_2}+(mn+2c)(t_2-t_0)-{n\over m}\Bigl({1\over t_0}-{1\over t_2}\Bigr)-2.
\end{equation}
Similarly, we have
\[
{T_{n+1}^2\over T_{n}^2}
={mt_0+1\over mt_1+1}e^B\biggl(1+O\Bigl({\log^{7/2}{n}\over\sqrt{n}}\Bigr)\biggr),
\]
where
\begin{equation}\label{pf:B}
B=2mt_1+{2\over mt_1}+(2mn+4c)(t_1-t_0)-{2n\over m}\Bigl({1\over t_0}-{1\over t_1}\Bigr)-2.
\end{equation}
Therefore, 
\begin{equation}\label{75}
{T_{n+2}\over T_{n}}-{T_{n+1}^2\over T_{n}^2}
=\left(\sqrt{mt_0+1\over mt_2+1}e^A-{mt_0+1\over mt_1+1}e^B\right)
\biggl(1+O\Bigl({\log^{7/2}{n}\over\sqrt{n}}\Bigr)\biggr).
\end{equation}
By the estimates~\eqref{61} and~\eqref{62}, we have
\begin{align*}
{mt_0+1\over mt_1+1}
&=1+O\Bigl({1\over n\log{n}}\Bigr),\\[5pt]
\sqrt{mt_0+1\over mt_2+1}
&=1+O\Bigl({mt_0+1\over mt_2+1}-1\Bigr)
=1+O\Bigl({1\over n\log{n}}\Bigr).
\end{align*}
Substituting the above estimates into~\eqref{75}, we get
\begin{equation}\label{72}
{T_{n+2}\over T_{n}}-{T_{n+1}^2\over T_{n}^2}
=\bigl(e^A-e^B\bigr)\biggl(1+O\Bigl({\log^{7/2}{n}\over\sqrt{n}}\Bigr)\biggr).
\end{equation}
By Cauchy's mean value theorem, there exists a constant $C$ such that $B<C<A$ and
\begin{equation}\label{73}
e^A-e^B=(A-B)e^C.
\end{equation}
We can compute directly from~\eqref{pf:A} and~\eqref{pf:B} that
\begin{equation}\label{pf:A-B}
A-B=2m(t_2-t_1)-(mn+2c)(2t_1-t_0-t_2)+{2\over m}\Bigl({1\over t_2}-{1\over t_1}\Bigr)
+{n\over m}\Bigl({1\over t_0}+{1\over t_2}-{2\over t_1}\Bigr).
\end{equation}
By Lemma~\ref{lem1}, we deduce that
\begin{align*}
2m(t_2-t_1)&={2\over n}+O\Bigl({1\over n\log{n}}\Bigr),\\[5pt]
-(mn+2c)(2t_1-t_0-t_2)&=-{1\over n}+O\Bigl({1\over n\log{n}}\Bigr),\\[5pt]
{2\over m}\Bigl({1\over t_2}-{1\over t_1}\Bigr)
&=O\Bigl({1\over n\log^2{n}}\Bigr),\\[5pt]
{n\over m}\Bigl({1\over t_0}+{1\over t_2}-{2\over t_1}\Bigr)
&=O\Bigl({1\over n\log^2{n}}\Bigr).
\end{align*}
Substituting them into~\eqref{pf:A-B}, we get
\begin{equation}\label{74}
A-B={1\over n}\biggl(1+O\Bigl({1\over \log{n}}\Bigr)\biggr).
\end{equation}
Similarly, we can derive that
\begin{align*}
A&=2mt_2+O\Bigl({1\over \log{n}}\Bigr),\\[5pt]
B&=2mt_1+O\Bigl({1\over \log{n}}\Bigr).
\end{align*}
By the estimates~\eqref{sim:r}, we find that
\begin{align*}
e^A&=e^{2mt_2}\biggl(1+O\Bigl({1\over \log{n}}\Bigr)\biggr)
={m^2n^2\over\log^2{n}}\biggl(1+O\Bigl({\log\log{n}\over\log{n}}\Bigr)\biggr),\\
e^B&=e^{2mt_1}\biggl(1+O\Bigl({1\over \log{n}}\Bigr)\biggr)
={m^2n^2\over\log^2{n}}\biggl(1+O\Bigl({\log\log{n}\over\log{n}}\Bigr)\biggr).
\end{align*}
Since $B<C<A$, we have
\begin{equation}\label{e^C_M}
e^C={m^2n^2\over\log^2{n}}\biggl(1+O\Bigl({\log\log{n}\over\log{n}}\Bigr)\biggr).
\end{equation}
Substituting~\eqref{e^C_M} and~\eqref{74} into~\eqref{73}, we deduce that
\begin{equation}\label{71}
e^A-e^B={m^2n\over\log^2{n}}\biggl(1+O\Bigl({\log\log{n}\over\log{n}}\Bigr)\biggr).
\end{equation}
Substituting~\eqref{71} into~\eqref{72},
and in view of Theorem~\ref{thm_EV_M},
we obtain the approximation 
\[
V_n={1\over m^2}\biggl({T_{n+2}\over T_{n}}-{T_{n+1}^2\over T_{n}^2}\biggr)-{1\over m}
\sim {n\over\log^2{n}}.
\]
This completes the proof.
\end{proof}

\noindent{\bf Acknowledgments.}
The author is grateful to the anonymous referee 
for the suggestion of generalizing the original $c$-version problem to
the present $(c,m)$-version problem,
and for bringing Dowling's paper to attention.
Thanks are also given to an anonymous referee who pointed out 
that the proof for the asymptotic expression of $T_{n}$ follows directly 
from the theory of Hayman-admissibility.
This work was supported by
the National Natural Science Foundation of China
(Grant No.~$11101010$).

\end{document}